\documentclass[12pt,a4paper]{amsart}

\usepackage[utf8]{inputenc}
\usepackage{enumerate,bbm,titletoc,amssymb,amsthm,mathabx,graphicx}
\usepackage[usenames,dvipsnames]{xcolor}
\usepackage{tikz}
\usepackage[colorlinks,linkcolor=BrickRed,citecolor=OliveGreen,urlcolor=black,hypertexnames=true]{hyperref}

\DeclareMathOperator{\tr}{tr}

\DeclareMathOperator{\rk}{rk}
\DeclareMathOperator{\GL}{GL}
\DeclareMathOperator{\dom}{dom}
\DeclareMathOperator{\hdom}{hdom}
\DeclareMathOperator{\spa}{span}
\DeclareMathOperator{\opm}{M}
\DeclareMathOperator{\oph}{H}
\DeclareMathOperator{\QM}{QM}
\newcommand{\ve}{\varepsilon}

\newcommand{\N}{\mathbb{N}}

\newcommand{\R}{\mathbb{R}}
\newcommand{\C}{\mathbb{C}}
\newcommand{\cA}{\mathcal{A}}
\newcommand{\cB}{\mathcal{B}}
\newcommand{\cD}{\mathcal{D}}

\newcommand{\cO}{\mathcal{O}}

\newcommand{\cR}{\mathcal{R}}

\newcommand{\rr}{\mathbbm r}
\newcommand{\rs}{\mathbbm s}
\newcommand{\rc}{\mathbbm c}
\newcommand{\mm}{\mathbbm m}
\newcommand{\qq}{\mathbbm q}

\newcommand{\rv}{\mathbbm v}
\newcommand{\fX}{\mathfrak{X}}

\newcommand*{\mathlarger}[1]{\scalebox{1.35}{$#1$}}
\newcommand*{\mat}[1]{\opm_{#1}(\C)}

\newcommand*{\her}[1]{\oph_{#1}(\C)}
\newcommand{\cfm}{\opm_\infty(\C)}

\newcommand{\hV}{S}

\newcommand{\Langle}{\mathop{<}\!}
\newcommand{\Rangle}{\!\mathop{>}}

\newcommand{\px}{\C\!\Langle x\Rangle}
\newcommand{\pxx}{\C\!\Langle x,x^*\Rangle}

\makeatletter
\def\moverlay{\mathpalette\mov@rlay}
\def\mov@rlay#1#2{\leavevmode\vtop{
		\baselineskip\z@skip \lineskiplimit-\maxdimen
		\ialign{\hfil$#1##$\hfil\cr#2\crcr}}}
\makeatother

\newcommand{\plangle}{\moverlay{(\cr<}}
\newcommand{\prangle}{\moverlay{)\cr>}}
\newcommand{\rx}{\C\plangle x \prangle}
\newcommand{\rex}{\mathfrak{R}_{\C}(x)}

\def\un{\operatorname{U}_d(\C)}

\newtheorem{thm}{Theorem}[section]
\newtheorem{lem}[thm]{Lemma}
\newtheorem{cor}[thm]{Corollary}
\newtheorem{prop}[thm]{Proposition}
\newtheorem{thmA}{Theorem}

\theoremstyle{definition}

\newtheorem{exa}[thm]{Example}
\theoremstyle{remark}
\newtheorem{rem}[thm]{Remark}

\makeatletter
\@namedef{subjclassname@2020}{
	\textup{2020} Mathematics Subject Classification}
\makeatother

\numberwithin{equation}{section}

\title{Hilbert's 17th problem in free skew fields}

\author[J. Vol\v{c}i\v{c}]{Jurij Vol\v{c}i\v{c}${}^\star$}
\address{Jurij Vol\v{c}i\v{c}, Department of Mathematics, Texas A\&M University}
\email{volcic@math.tamu.edu}
\thanks{${}^\star$Supported by the NSF grant DMS 1954709.}

\date{\today}

\begin{document}
	
\begin{abstract}
This paper solves the rational noncommutative analog of Hilbert's 17th problem: if a noncommutative rational function is positive semidefinite on all tuples of hermitian matrices in its domain, then it is a sum of hermitian squares of noncommutative rational functions. This result is a generalization and culmination of earlier positivity certificates for noncommutative polynomials or rational functions without hermitian singularities. More generally, a rational Positivstellensatz for free spectrahedra is given: a noncommutative rational function is positive semidefinite or undefined at every matricial solution of a linear matrix inequality $L\succeq0$ if and only if it belongs to the rational quadratic module generated by $L$. The essential intermediate step towards this Positivstellensatz for functions with singularities is an extension theorem for invertible evaluations of linear matrix pencils.
\end{abstract}

\maketitle


\section{Introduction}

In his famous problem list of 1900, Hilbert asked whether every positive rational function can be written as a sum of squares of rational functions. The affirmative answer by Artin in 1927 laid ground for the rise of real algebraic geometry \cite{BCR}. Several other sum-of-squares certificates (Positivstellens\"atze) for positivity on semialgebraic sets followed; since the detection of sums of squares became viable with the emergence of semidefinite programming \cite{WSV}, these certificates play a fundamental role in polynomial optimization \cite{Las,BPT}.

Positivstellens\"atze are also essential in the study of polynomial and rational inequalities in matrix variables, which splits into two directions.
The first one deals with inequalities where the size of the matrix arguments is fixed \cite{PS,KSV}.
The second direction attempts to answer questions about positivity of noncommutative polynomials and rational functions when matrix arguments of all finite sizes are considered. Such questions naturally arise in control systems \cite{dOHMP}, operator algebras \cite{Oza} and quantum information theory \cite{DLTW,many}.
This (dimension-)free real algebraic geometry started with the seminal work of Helton \cite{Hel} and McCullough \cite{McC}, who proved that a noncommutative polynomial is positive semidefinite on all tuples of hermitian matrices precisely when it is a sum of hermitian squares of noncommutative polynomials.
The purpose of this paper is to extend this result to noncommutative rational functions.

Let $x=(x_1,\dots,x_d)$ be freely noncommuting variables. The free algebra $\px$ of noncommutative polynomials admits a universal skew field of fractions $\rx$, also called the free skew field \cite{Coh,CR}, whose elements are noncommutative rational functions. We endow $\rx$ with the unique involution $*$ that fixes the variables and conjugates the scalars. One can consider positivity of noncommutative rational functions on tuples of hermitian matrices. For example, let
$$\rr=x_3^2+x_4^2-(x_3x_1+x_4x_2)(x_1^2+x_2^2)^{-1}(x_1x_3+x_2x_4)\in\rx.$$
It turns out $\rr(X)$ is a positive semidefinite matrix for every tuple of hermitian matrices $X=(X_1,X_2,X_3,X_4)$ belonging to the domain of $\rr$ 
(meaning $\ker X_1\cap\ker X_2=\{0\}$ in this particular case). 
One way to certify this is by observing that $\rr=\rr_1\rr_1^*+\rr_2\rr_2^*$ where
$$\rr_1=(x_4-x_3x_1^{-1}x_2)x_2(x_1^2+x_2^2)^{-1}x_1, \qquad
\rr_2=(x_4-x_3x_1^{-1}x_2)(1+x_2x_1^{-2}x_2)^{-1}.$$
The solution of Hilbert's 17th problem in the free skew field presented in this paper (Corollary \ref{c:H17}) states that every $\rr\in\rx$, positive semidefinite on its hermitian domain, is a sum of hermitian squares in $\rx$.
This statement was proved in \cite{KPV} for noncommutative rational functions $\rr$ that are {\it regular}, meaning that $\rr(X)$ is well-defined for every tuple of hermitian matrices. As with most noncommutative Positivstellens\"atze, at the heart of this result is a variation of the Gelfand-Naimark-Segal (GNS) construction. Namely, if $\rr\in\rx$ is not a sum of hermitian squares, one can construct a tuple of finite-dimensional hermitian operators $Y$ that is a sensible candidate for witnessing non-positive-definiteness of $\rr$. However, the construction itself does not guarantee that $Y$ actually belongs to the domain of $\rr$. This is not a problem if one assumes that $\rr$ is regular, as it was done in \cite{KPV}. However, it is worth mentioning that deciding regularity of a noncommutative rational function is a challenge on its own, as observed in \cite{KPV}. In this paper, the domain issue is resolved with an extension result: the tuple $Y$ obtained from the GNS construction can be extended to a tuple of finite-dimensional hermitian operators in the domain of $\rr$ without losing the desired features of $Y$.

The first main theorem of this paper pertains to linear matrix pencils and is key for the extension mentioned above. It might also be of independent interest in the study of quiver representations and semi-invariants \cite{Kin,DM}. Let $\otimes$ denote the Kronecker product of matrices.

\begin{thmA}\label{ta:1}
Let $\Lambda\in\mat{e}^d$ be such that $\Lambda_1\otimes X_1+\cdots+\Lambda_d\otimes X_d$ is invertible for some $X\in\mat{k}^d$. 
If $Y\in\mat{\ell}^d$, $Y'\in\mat{m\times \ell}^d$, $Y''\in\mat{\ell\times m}^d$ are such that
$$
\Lambda_1\otimes \begin{pmatrix} Y_1 \\ Y'_1\end{pmatrix}+\cdots+
\Lambda_d\otimes \begin{pmatrix} Y_d \\ Y'_d\end{pmatrix}
\qquad\text{and}\qquad
\Lambda_1\otimes \begin{pmatrix} Y_1 & Y''_1\end{pmatrix}+\cdots+
\Lambda_d\otimes \begin{pmatrix} Y_d & Y''_d\end{pmatrix}
$$
have full rank, then there exists $Z\in\mat{n}^d$ for some $n\ge m$ such that
$$
\Lambda_1\otimes \left(\begin{array}{cc}
Y_1 & \begin{matrix} Y''_1 & 0\end{matrix} \\
\begin{matrix} Y'_1 \\ 0\end{matrix} & \mathlarger{Z_1}
\end{array}\right)
+\cdots+
\Lambda_d\otimes \left(\begin{array}{cc}
Y_d & \begin{matrix} Y''_d & 0\end{matrix} \\
\begin{matrix} Y'_d \\ 0\end{matrix} & \mathlarger{Z_d}
\end{array}\right)
$$
is invertible.
\end{thmA}

See Theorem \ref{t:nonherm} for the proof. Together with a truncated rational imitation of the GNS construction, Theorem \ref{ta:1} leads to a rational Positivstellensatz on free spectrahedra. Given a monic hermitian pencil $L=I+H_1x_1+\cdots +H_dx_d$, the associated free spectrahedron $\cD(L)$ is the set of hermitian tuples $X$ satisfying the linear matrix inequality $L(X)\succeq0$. Since every convex solution set of a noncommutative polynomial is a free spectrahedron \cite{HM1}, the following statement is called a rational convex Positivstellensatz, and it generalizes its analogs in the polynomial context \cite{HKM} and regular rational context \cite{Pas}.

\begin{thmA}\label{ta:2}
Let $L$ be a hermitian monic pencil and $\rr\in\rx$. Then $\rr\succeq0$ on $\cD(L)\cap\dom\rr$ if and only if $\rr$ belongs to the  rational quadratic module generated by $L$:
$$\rr=\rr_1^*\rr_1+\cdots+\rr_m^*\rr_m+\rv_1^*L\rv_1+\cdots+\rv_n^*L\rv_n$$
where $\rr_i\in\rx$ and $\rv_j$ are vectors over $\rx$.
\end{thmA}

A more precise quantitative version is given in Theorem \ref{t:main} and has several consequences. The solution of Hilbert's 17th problem in $\rx$ is obtained by taking $L=1$ in Corollary \ref{c:H17}. Versions of Theorem \ref{ta:2} for invariant (Corollary \ref{c:inv}) and real (Corollary \ref{c:real}) noncommutative rational functions are also given. Furthermore, it is shown that the rational Positivstellensatz also holds for a family of quadratic polynomials describing non-convex sets (Subsection \ref{ss:noncvx}). As a contribution to optimization, Theorem \ref{ta:2} implies that the eigenvalue optimum of a noncommutative rational function on a free spectrahedron can be obtained by solving a single semidefinite program (Subsection \ref{ss:opt}), much like in the noncommutative polynomial case \cite{BPT,BKP} (but not in the classical commutative setting). 

Finally, Section \ref{s6} contains complementary results about domains of noncommutative rational functions. It is shown that every $\rr\in\rx$ can be represented by a formal rational expression that is well-defined at every hermitian tuple in the domain of $\rr$ (Proposition \ref{p:rep}); this statement fails in general if arbitrary matrix tuples are considered. On the other hand, a Nullstellensatz for cancellation of non-hermitian singularities is given in Proposition \ref{p:canc}.

\subsection*{Acknowledgment}
The author thanks Igor Klep for valuable comments and suggestions which improved the presentation of this paper.

\section{Preliminaries}\label{s2}

In this section we establish terminology, notation and preliminary results on noncommutative rational functions that are used throughout the paper. Let $\mat{m\times n}$ denote the space of complex $m\times n$ matrices, and $\mat{n}=\mat{n\times n}$. Let $\her{n}$ denote the real space of hermitian $n\times n$ matrices. For $X=(X_1,\dots,X_d)\in \mat{m\times n}^d$, $A\in\mat{p\times m}$ and $B\in\mat{n\times q}$ we write 
$$AXB= (AX_1B,\dots,AX_dB)\in \mat{p\times q}^d, \qquad X^*=(X_1^*,\dots,X_d^*)\in\mat{n\times m}^d.$$

\subsection{Free skew field}

We define noncommutative rational functions using formal rational expressions and their matrix evaluations as in \cite{KVV2}. {\bf Formal rational expressions} are syntactically valid combinations of scalars, freely noncommuting variables $x=(x_1,\dots,x_d)$, rational operations and parentheses. More precisely, a formal rational expression is an ordered (from left to right) rooted tree whose leaves have labels from $\C\cup\{x_1,\dots,x_d\}$, and every other node is either labeled $+$ or $\times$ and has two children, or is labeled ${}^{-1}$ and has one child. For example, $((2+x_1)^{-1}x_2)x_1^{-1}$ is a formal rational expression corresponding to the ordered tree

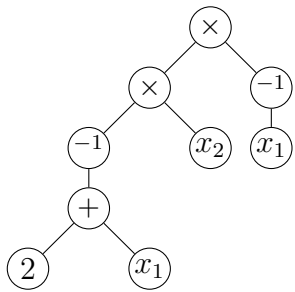
\begin{center}
\begin{tikzpicture}[scale=0.8,every node/.style={draw=black,circle,inner sep=0pt,minimum size=3ex}]
\node (a) at (0,0) {$2$};
\node (b) at (2,0) {$x_1$};
\node (c) at (1,1) {$+$};
\node (d) at (1,2) {${}^{-1}$};
\node (e) at (3,2) {$x_2$};
\node (f) at (2,3) {$\times$};
\node (g) at (4,3) {${}^{-1}$};
\node (h) at (4,2) {$x_1$};
\node (i) at (3,4) {$\times$};
\path [-] (a) edge (c);
\path [-] (b) edge (c);
\path [-] (c) edge (d);
\path [-] (d) edge (f);
\path [-] (e) edge (f);
\path [-] (f) edge (i);
\path [-] (h) edge (g);
\path [-] (g) edge (i);
\end{tikzpicture}
\end{center}

A {\bf subexpression} of a formal rational expression $r$ is any formal rational expression which appears in the construction of $r$ (i.e., as a sub-tree). For example, all subexpressions of $((2+x_1)^{-1}x_2)x_1^{-1}$ are
$$2,\ x_1,\ 2+x_1,\ (2+x_1)^{-1},\ x_2,\ (2+x_1)^{-1}x_2,\ x_1^{-1},\ ((2+x_1)^{-1}x_2)x_1^{-1}.$$

Given a formal rational expression $r$ and $X\in \mat{n}^d$, the evaluation $r(X)$ is defined in the natural way if all inverses appearing in $r$ exist at $X$. The set of all $X\in\mat{n}^d$ such that $r$ is defined at $X$ is denoted $\dom_n r$. The {\bf (matricial) domain} of $r$ is
$$\dom r = \bigcup_{n\in\N} \dom_n r.$$
Note that $\dom_n r$ is a Zariski open set in $\mat{n}^d$ for every $n\in\N$. A formal rational expression $r$ is {\bf non-degenerate} if $\dom r\neq\emptyset$; let $\rex$ denote the set of all non-degenerate formal rational expressions. On $\rex$ we define an equivalence relation $r_1\sim r_2$ if and only if $r_1(X)=r_2(X)$ for all $X\in\dom r_1\cap \dom r_2$. Equivalence classes with respect to this relation are called {\bf noncommutative  rational functions}. By \cite[Proposition 2.2]{KVV2} they form a skew field denoted $\rx$, which is the universal skew field of fractions of the free algebra $\px$ by \cite[Section 4.5]{Coh}. The equivalence class of $r\in\rex$ is denoted $\rr\in\rx$; we also write $r\in\rr$ and say that $r$ is a representative of the noncommutative rational function $\rr$.

There is a unique involution $*$ on $\rx$ that is determined by $\alpha^*=\overline{\alpha}$ for $\alpha\in\C$ and $x_j^*=x_j$ for $j=1,\dots,d$. Furthermore, this involution lifts to an involutive map $*$ on the set $\rex$: in terms of ordered trees, $*$ transposes a tree from left to right and conjugates the scalar labels. Note that $X\in\dom r$ implies $X^*\in\dom r^*$ for $r\in \rex$. 

\subsection{Hermitian domain}

For $r\in\rex$ let $\hdom_n r= \dom_n r\cap \her{n}^d$. Then
$$\hdom r = \bigcup_{n\in\N} \hdom_n r$$
is the {\bf hermitian domain} of $r$. Note that $\hdom_n r$ is Zariski dense in $\dom_n r$ because $\her{n}$ is Zariski dense in $\mat{n}$ and $\dom_n r$ is Zariski open in $\mat{n}^d$. Finally, we define the (hermitian) domain of a noncommutative rational function: for $\rr\in\rx$ let
$$\dom\rr = \bigcup_{r\in\rr} \dom r,\qquad \hdom\rr = \bigcup_{r\in\rr} \hdom r.$$
By the definition of the equivalence relation on non-degenerate expressions, $\rr$ has a well-defined evaluation at $X\in\dom\rr$, written as $\rr(X)$, which equals $r(X)$ for any representative $r$ of $\rr$ that has $X$ in its domain. The following proposition is a generalization of \cite[Proposition 3.3]{KPV} and is proved in Subsection \ref{ss:rep}.

\begin{prop}\label{p:rep}
For every $\rr\in\rx$ there exists $r\in\rr$ such that $\hdom \rr=\hdom r$.
\end{prop}

\begin{rem}\label{r:rep}
There are noncommutative rational functions such that $\dom \rr\neq\dom r$ for every $r\in\rr$, see Example \ref{ex:hdom} or \cite[Example 3.13]{Vol}.
\end{rem}

\subsection{Linear representation of a formal rational expression}

A fundamental tool for handling noncommutative rational functions are linear representations (also linearizations or realizations) \cite{CR,Coh,HMS}. Let $r\in\rex$. By \cite[Theorem 4.2 and Algorithm 4.3]{HMS} there exist $e\in\N$, vectors $u,v\in\C^e$ and an affine matrix pencil $M=M_0+M_1x_1+\cdots+M_dx_d$ with $M_j\in\mat{e}$ satisfying the following. For every unital $\C$-algebra $\cA$ and $a\in\cA^d$,
\begin{enumerate}[(i)]
	\item if $r$ can be evaluated at $a$, 
	then $M(a)\in\GL_e(\cA)$ and
	$r(a) = u^* M(a)^{-1}v$;
	\item if $M(a)\in\GL_e(\cA)$ and $\cA=\mat{n}$ for some $n\in\N$, then $r$ can be evaluated at $a$.
\end{enumerate}
We say that the triple $(u,M,v)$ is a {\bf linear representation} of $r$ of size $e$. Usually, linear representations are defined for noncommutative rational functions and with less emphasis on domains; however, the definition above is more convenient for the purpose of this paper.

\begin{rem}
In the definition of a linear representation, (ii) is valid not only for $\mat{n}$ but more broadly for stably finite algebras \cite[Lemma 5.2]{HMS}. However, it may fail in general, e.g. for the algebra of all bounded operators on an infinite-dimensional Hilbert space.
\end{rem}

We will also require the following proposition on pencils that is a combination of various existing results.

\begin{prop}[{\cite{Coh,KVV2,DM}}]\label{p:others}
Let $M$ be an affine pencil of size $e$. The following are equivalent:
\begin{enumerate}[(i)]
	\item $M\in\GL_e(\rx)$;
	\item there are $n\in\N$ and $X\in\mat{n}^d$ such that $\det M(X)\neq0$;
	\item for every $n\ge e-1$ there is $X\in\mat{n}^d$ such that $\det M(X)\neq0$;
	\item if $U\in \mat{e'\times e}$, $V\in\mat{e\times e''}$ satisfy $UMV=0$, 
	then $\rk U+\rk V\le e$.
\end{enumerate}
\end{prop}

\begin{proof}
(i)$\Leftrightarrow$(ii) follows by the construction of the free skew field via matrix evaluations (cf. \cite[Proposition 2.1]{KVV2}). (iii)$\Rightarrow$(ii) is trivial and (ii)$\Rightarrow$(iii) holds by \cite[Theorem 1.8]{DM}.
(iv)$\Leftrightarrow$(i) follows from \cite[Corollaries 4.5.9 and 6.3.6]{Coh} since the free algebra $\px$ is a free ideal ring \cite[Theorem 5.4.1]{Coh}.
\end{proof}

An affine matrix pencil is {\bf full} \cite[Section 1.4]{Coh} if it satisfies the (equivalent) properties in Proposition \ref{p:others}.

\begin{rem}\label{r:dom}
If $r\in\rex$ admits a linear representation of size $e$, then $\hdom_n r\neq\emptyset$ for $n\ge e-1$ by Proposition \ref{p:others} and the Zariski denseness of $\hdom_n r$ in $\dom_n r$.
\end{rem}

\section{An extension theorem}\label{s3}

An affine matrix pencil $M$ of size $e$ is {\bf irreducible} if $UMV=0$ for nonzero matrices $U\in \mat{e'\times e}$ and $V\in\mat{e\times e''}$ implies $\rk U+\rk V\le e-1$. In other words, a pencil is not irreducible if
it can be put into a $2\times2$ block upper triangular form with square diagonal blocks
$(\begin{smallmatrix}\star & \star \\ 0 & \star\end{smallmatrix})$ 
by a left and a right basis change.
Every irreducible pencil is full. On the other hand, every full pencil is, up to a left and a right basis change, equal to a block upper-triangular pencil whose diagonal blocks are irreducible pencils. 
In terms of quiver representations \cite{Kin}, $M=M_0+\sum_{j=1}^dM_jx_j$ is full/irreducible if and only if the $(e,e)$-dimensional representation $(M_0,M_1,\dots,M_d)$ of the $(d+1)$-Kronecker quiver is $(1,-1)$-semistable/stable.

For the purpose of this section we extend evaluations of linear matrix pencils to tuples of rectangular matrices.
If $\Lambda=\sum_{j=1}^d \Lambda_jx_j$ is of size $e$ and $X\in\mat{\ell \times m}^d$ then
$$\Lambda(X)=\sum_{j=1}^d \Lambda_j\otimes X_j \in \mat{e\ell\times em}.$$

The following lemma and proposition rely on an ampliation trick in a free algebra to demonstrate the existence of specific invertible evaluations of full pencils (see \cite[Section 2.1]{HKV} for another argument involving such ampliations).

\begin{lem}\label{l:irr}
Let $\Lambda = \sum_{j=1}^d\Lambda_jx_j$ be a homogeneous irreducible pencil of size $e$. Let $\ell\le m$ and denote $n=(m-\ell)(e-1)$. Given $C\in\mat{me\times \ell e}$, consider the pencil $\widetilde{\Lambda}$ of size $(m+n)e$ in $d(m+n)(n+m-\ell)$ variables $z_{jpq}$,
$$\widetilde{\Lambda} = 
\begin{pmatrix}
C & 0 \\ 0 & 0
\end{pmatrix}
+\sum_{j=1}^d\sum_{q=1}^{n+m-\ell}\left(
\sum_{p=1}^m
\begin{pmatrix}
0 & \widehat{E}_{p,q}\otimes \Lambda_j \\ 0 & 0
\end{pmatrix}z_{jpq}
+\sum_{p=m+1}^{m+n}
\begin{pmatrix}
0 & 0 \\ 0 & \widecheck{E}_{p-m,q}\otimes \Lambda_j
\end{pmatrix}z_{jpq}
\right)
$$
where $\widehat{E}_{p,q} \in\mat{m\times (n+m-\ell)}$ and $\widecheck{E}_{p-m,q} \in\mat{n\times (n+m-\ell)}$ are the standard matrix units. If $C$ has full rank, then the pencil $\widetilde{\Lambda}$ is full.
\end{lem}

\begin{proof}
Suppose $U$ and $V$ are constant matrices with $e(m+n)$ columns and $e(m+n)$ rows, respectively, that satisfy $U\widetilde{\Lambda}V=0$. There is nothing to prove if $U=0$, so let $U\neq0$. Write
$$U=\begin{pmatrix} U_1 & \cdots & U_{m+n}\end{pmatrix},
\qquad
V=\begin{pmatrix} V_0 \\ V_1 \\ \vdots \\ V_{n+m-\ell}\end{pmatrix}$$
where each $U_p$ has $e$ columns, $V_0$ has $\ell e$ rows, and each $V_q$ with $q>0$ has $e$ rows. Also let $U_0=(U_1 \ \cdots \ U_m)$. Then $U\widetilde{\Lambda}V=0$ implies
\begin{align}
\label{e1} U_0CV_0 & = 0, \\
\label{e2} U_p\Lambda V_q & = 0 \qquad \text{for}\quad 1\le p\le m+n,\ 1\le q\le n+m-\ell.
\end{align}
Since $C$ has full rank, \eqref{e1} implies $\rk U_0+\rk V_0\le me$. Note that $U_{p'}\neq0$ for some $1\le p'\le m+n$ because $U\neq0$. Since $\Lambda$ is irreducible and $U_{p'}\neq 0$ for some $p'$, \eqref{e2} implies $\rk V_q\le e-1$ and $\rk U_p+\rk V_q\le e-1$ for all $p,q>0$. Then
\begin{align*}
\rk U+\rk V
& \le \rk U_0+\rk V_0+\sum_{p=m+1}^{m+n} \rk U_p+\sum_{q=1}^{n+m-\ell}\rk V_q \\
& \le me+n(e-1)+(m-\ell)(e-1) \\
& = (m+n)e 
\end{align*}
by the choice of $n$. Therefore $\widetilde{\Lambda}$ is full.
\end{proof}

\begin{prop}\label{p:side}
Let $\Lambda$ be a homogeneous full pencil of size $e$, and let $X\in\mat{m\times \ell}^d$ with $\ell\le m$ be such that $\Lambda(X)$ has full rank. Then there exist $\widehat{X}\in \mat{m\times (n+m-\ell)}^d$ and $\widecheck{X}\in \mat{n\times (n+m-\ell)}^d$ for some $n\in\N$ such that
\begin{equation}\label{e:side}
\det \Lambda
\begin{pmatrix}
X & \widehat{X} \\ 0 & \widecheck{X}
\end{pmatrix}
\neq0.
\end{equation}
\end{prop}

\begin{proof}
A full pencil is up to a left-right basis change equal to a block upper-triangular pencil with irreducible diagonal blocks. Suppose that the lemma holds for irreducible pencils; since the set of pairs
$(\widehat{X},\widecheck{X})\in \mat{m\times (n+m-\ell)}^d\times \mat{n\times (n+m-\ell)}^d$ satisfying \eqref{e:side} is Zariski open, the lemma then also holds for full pencils. Thus we can without loss of generality assume that $\Lambda$ is irreducible.

Let $n_1=(m-\ell)(e-1)$ and $e_1=(m+n_1)e$. By Lemma \ref{l:irr} applied to $C=\sum_{j=1}^d X_j\otimes \Lambda_j$ and Proposition \ref{p:others}, there exists
$Z \in\mat{e_1-1}^{d(m+n_1)(n_1+m-\ell)}$ such that $\widetilde{\Lambda}(Z)$ is invertible. Therefore the matrix
$$
\sum_{j=1}^d\Lambda_j\otimes\left(
\left(\begin{smallmatrix}I\otimes X_j & 0 \\ 0 & 0\end{smallmatrix}\right)
+\sum_{q=1}^{n_1+m-\ell}\left(\sum_{p=1}^{n_1}
\left(\begin{smallmatrix}
0 & Z_{jpq}\otimes \widehat{E}_{p,q} \\ 0 & 0
\end{smallmatrix}\right)
+\sum_{p=m+1}^{m+n_1}
\left(\begin{smallmatrix}
0 & 0 \\ 0 & Z_{jpq}\otimes \widecheck{E}_{p-m,q}
\end{smallmatrix}\right)
\right)\right)
$$
is invertible since it is similar to $\widetilde{\Lambda}(Z)$ (via a permutation matrix).
Thus there are $\widehat{X}\in \mat{m\times (n+m-\ell)}^d$ and $\widecheck{X}\in \mat{n\times (n+m-\ell)}^d$ such that
$$\det \Lambda
\begin{pmatrix}
	X & \widehat{X} \\ 0 & \widecheck{X}
\end{pmatrix}
\neq0$$
where $n=(e_1-2)m+n_1(e_1-1)$. 
\end{proof}

We are ready to prove the first main result of the paper.

\begin{thm}\label{t:nonherm}
Let $\Lambda$ be a full pencil of size $e$, and let $Y\in\mat{\ell}^d$, $Y'\in\mat{m\times \ell}^d$, $Y''\in\mat{\ell\times m}^d$ be such that
\begin{equation}\label{e:rk}
\Lambda \begin{pmatrix} Y \\ Y'\end{pmatrix},\qquad
\Lambda \begin{pmatrix} Y & Y''\end{pmatrix}
\end{equation}
have full rank. Then there are $n\ge m$ and $Z\in\mat{n}^d$
such that
$$
\det \Lambda 
\left(\begin{array}{cc}
Y & \begin{matrix} Y'' & 0\end{matrix} \\
\begin{matrix} Y' \\ 0\end{matrix} & \mathlarger{Z}
\end{array}\right)
\neq0.
$$
\end{thm}

\begin{proof}
By Proposition \ref{p:side} and its transpose analog there exist $k\in\N$ and
\begin{align*}
A'\in \mat{\ell\times (k+m-\ell)}^d,\quad 
B'\in \mat{m\times (k+m-\ell)}^d,\quad
C'\in \mat{k\times (k+m-\ell)}^d, \\
A''\in \mat{(k+m-\ell)\times \ell}^d,\quad 
B''\in \mat{(k+m-\ell)\times m}^d,\quad
C''\in \mat{(k+m-\ell)\times k}^d \\
\end{align*}
such that the matrices
$$\Lambda
\begin{pmatrix}
Y & A' \\ Y' & B' \\ 0 & C'
\end{pmatrix}
,\qquad
\Lambda
\begin{pmatrix}
Y & Y'' &0 \\ A'' & B'' & C''
\end{pmatrix}
$$
are invertible. Consequently there exists $\ve\in\C\setminus\{0\}$ such that
$$\begin{pmatrix}
\Lambda(Y) & 0 & 0 &0 & 0 \\
0 & 0 & 0 &0 & 0 \\
0 & 0 & 0 &0 & 0 \\
0 & 0 & 0 &0 & 0 \\
0 & 0 & 0 &0 & 0
\end{pmatrix}
+\ve\begin{pmatrix}
\begin{pmatrix}0 & 0 \\ 0 &0 \end{pmatrix} &
\Lambda
\begin{pmatrix}
Y & Y'' &0 \\ A'' & B'' & C''
\end{pmatrix}
\\
\Lambda
\begin{pmatrix}
Y & A' \\ Y' & B' \\ 0 & C'
\end{pmatrix}
&
\begin{pmatrix}0 & 0 & 0\\ 0 &0 &0\\ 0 &0 &0\end{pmatrix}
\end{pmatrix}$$
is invertible; this matrix is similar to
\begin{equation}\label{e:eps}
\Lambda
 \begin{pmatrix}
 Y & 0 & \ve Y & \ve Y'' &0 \\
 0 & 0 & \ve A''& \ve B'' & \ve C'' \\
 \ve Y  & \ve A' & 0 &0 &0 \\
 \ve Y' & \ve B' & 0 &0 &0 \\
 0 & \ve C' & 0 &0 &0
 \end{pmatrix}.
\end{equation}
Thus the matrix \eqref{e:eps} is invertible; its block structure and the linearity of $\Lambda$ imply that \eqref{e:eps} is invertible for every $\ve\neq0$, so we can choose $\ve=1$. After performing elementary row and column operations on \eqref{e:eps} we conclude that
\begin{equation}\label{e:mat5}
\Lambda
 \begin{pmatrix}
Y & Y'' & 0 & 0 &0 \\
 Y' & 0 & - Y' & B' &0 \\
0  & - Y'' & - Y &  A' &0 \\
0 & B'' &  A''& 0 & C'' \\
0 & 0 & 0 & C' &0
\end{pmatrix}
\end{equation}
is invertible. So the lemma holds for $n= 2(m+k)$.
\end{proof}

\begin{rem}\label{r:size}
It follows from the proofs of Proposition \ref{p:side} and Theorem \ref{t:nonherm} that one can choose
$$n= 2 \left(e^3 m^2+e m (2 e \ell-1)+\ell (e \ell-2)\right)$$
in Theorem \ref{t:nonherm}. However, this is unlikely to be the minimal choice for $n$.
\end{rem}

Let $\cfm$ be the algebra of $\N\times\N$ matrices over $\C$ that have only finitely many nonzero entries in each column; that is, elements of $\cfm$ can be viewed as operators on $\oplus^{\N}\C$. Given $r\in\rex$ let $\dom_\infty r$ be the set of tuples $X\in\cfm^d$ such that $r(X)$ is well-defined. If $(u,M,v)$ is a linear representation of $r$ of size $e$, then $M(X)\in\opm_e(\cfm)$ is invertible for every $X\in\dom_\infty r$ by the definition of a linear representation adopted in this paper.

\begin{prop}\label{p:main}
Let $r\in\rex$. If $X\in\her{\ell}^d$ and $Y\in\mat{m\times \ell}^d$ are such that
$$\begin{pmatrix}
X & \begin{matrix} Y^* & 0\end{matrix} \\
\begin{matrix} Y \\ 0\end{matrix} & \mathlarger{W}
\end{pmatrix} 
\in \dom_\infty r
$$
for some $W\in\cfm^d$, then there exist $n\ge m$, $E\in\mat{n}$ and $Z\in\her{n}^d$ such that
$$\begin{pmatrix}
X & \begin{pmatrix} Y^* & 0\end{pmatrix}E^* \\
E\begin{pmatrix} Y \\ 0\end{pmatrix} & \mathlarger{Z}
\end{pmatrix} \in \hdom r.
$$
\end{prop}

\begin{proof}
Let $(u,M,v)$ be a linear representation of $r$ of size $e$.
By assumption,
$$M\begin{pmatrix}
X & \begin{matrix} Y^* & 0\end{matrix} \\
\begin{matrix} Y \\ 0\end{matrix} & \mathlarger{W}
\end{pmatrix} 
$$
is an invertible matrix over $\cfm$. If $M=M_0+M_1x_1+\cdots+M_dx_d$, then the matrices
$$M_0\otimes \begin{pmatrix}I \\ 0\end{pmatrix}+
\sum_{j=1}^d M_0\otimes \begin{pmatrix}X_j \\ Y_j\end{pmatrix},\qquad
M_0\otimes \begin{pmatrix}I & 0\end{pmatrix}+
\sum_{j=1}^d M_0\otimes \begin{pmatrix}X_j & Y_j^*\end{pmatrix}
$$
have full rank. Let $n\in\N$ be as in Theorem \ref{t:nonherm}. Then there is $Z'\in \mat{n}^{1+d}$ such that
\begin{equation}\label{e:2psd}
\det\left(M_0\otimes
\begin{pmatrix}
I & \begin{matrix} 0 & 0\end{matrix} \\
\begin{matrix} 0 \\ 0\end{matrix} & \mathlarger{Z'_0}
\end{pmatrix}+
\sum_{j=1}^d M_0\otimes 
\begin{pmatrix}
X_j & \begin{matrix} Y_j^* & 0\end{matrix} \\
\begin{matrix} Y_j \\ 0\end{matrix} & \mathlarger{Z'_j}
\end{pmatrix}
\right)\neq0
\end{equation}
is invertible. The set of all $Z'\in\mat{n}^{1+d}$ satisfying \eqref{e:2psd} is thus a nonempty Zariski open set in $\mat{n}^{1+d}$. Since the set of positive definite $n\times n$ matrices is Zariski dense in $\mat{n}$, there exists $Z'\in\her{n}^{1+d}$ with $Z'_0\succ0$ such that \eqref{e:2psd} holds. If $Z_0' = E^{-1}E^{-*}$, let $Z_j=EZ'_jE^*$ for $1\le j \le d$. Then
$$M\begin{pmatrix}
X & \begin{pmatrix} Y^* & 0\end{pmatrix}E^* \\
E\begin{pmatrix} Y \\ 0\end{pmatrix} & \mathlarger{Z}
\end{pmatrix}
$$
is invertible, so
$$
\begin{pmatrix}
X & \begin{pmatrix} Y^* & 0\end{pmatrix}E^* \\
E\begin{pmatrix} Y \\ 0\end{pmatrix} & \mathlarger{Z}
\end{pmatrix}  \in\hdom r 
$$
by the definition of a linear representation.
\end{proof}

We also record a non-hermitian version of Proposition \ref{p:main}.

\begin{prop}\label{p:main2}
Let $r\in\rex$. If $X\in\mat{m\times \ell}^d$ with $\ell\le m$ is such that
$$\begin{pmatrix}
\begin{matrix} X\\ 0\end{matrix} & \mathlarger{W}
\end{pmatrix} 
\in \dom_\infty r
$$
for some $W\in\cfm^d$, then there exist $n\ge m$ and $Z\in\mat{n\times (n-\ell)}^d$ such that
$$\begin{pmatrix}
\begin{matrix} X \\ 0\end{matrix} & \mathlarger{Z}
\end{pmatrix} \in \dom r.
$$
\end{prop}

\begin{proof}
We apply a similar reasoning as in the proof of Proposition \ref{p:main}, only this time with Proposition \ref{p:side} instead of Theorem \ref{t:nonherm}, and without hermitian considerations.
\end{proof}

\section{Multiplication operators attached to a formal rational expression}\label{s4}

In this section we assign a tuple of operators  $\fX$ on a vector space of countable dimension to each formal rational expression $r$, so that $r$ is well-defined at $\fX$ and the finite-dimensional restrictions of $\fX$ partially retain a certain multiplicative property.

Fix an expression $r\in\rex$. Without loss of generality we assume that all the variables in $x$ appear as subexpressions in $r$ (otherwise we replace $x$ by a suitable sub-tuple). Let
$$R=\{1\}\cup \{q\in\rex\setminus\C\colon q \text{ is a subexpression of }r \text{ or }r^*  \}\subset \rex.$$
Note that $R$ is finite, $\hdom q\supseteq\hdom r$ for $q\in R$, and $q\in R$ implies $q^*\in R$. Let $\cR\subset\rx$ be the set of noncommutative rational functions represented by $R$. 
For $\ell\in\N$ we define finite-dimensional vector subspaces 
$$V_\ell = \spa_\C\overbrace{\cR\cdots \cR}^\ell\subset \rx.$$
Note that $V_\ell\subseteq V_{\ell+1}$ since $1\in R$.
Furthermore, let $V=\bigcup_{\ell\in\N} V_\ell$. Then $V$ is a finitely generated $*$-subalgebra of $\rx$. For $j=1,\dots,d$ we define operators
$$\fX_j:V\to V,\qquad \fX_j \rs = x_j\rs.$$

\begin{lem}\label{l:fun}
There is a linear functional $\phi:V\to\C$ such that $\phi(\rs^*)=\overline{\phi(\rs)}$ and $\phi(\rs\rs^*)>0$ for all $\rs\in V\setminus\{0\}$.
\end{lem}

\begin{proof}
For some $X\in\hdom r$ let $m=\max_{q\in R}\|q(X)\|$. Let $\ell\in\N$. Since $V_\ell$ is finite-dimensional, there exist $n_\ell\in\N$ and $X^{(\ell)}\in\hdom_{n_\ell} r$ such that
\begin{equation}\label{e:pt}
\max_{q\in R}\|q(X^{(\ell)})\|\le m+1 \qquad \text{and}\qquad 
\rs(X^{(\ell)})\neq 0\ \text{ for all }\ \rs\in V_{\ell}\setminus\{0\}
\end{equation}
by the local-global linear dependence principle for noncommutative rational functions, see \cite[Theorem 6.5]{Vol2} or \cite[Corollary 8.87]{BPT}. Define
$$\phi:V\to\C,\qquad
\phi(\rs)=\sum_{\ell=1}^\infty \frac{1}{\ell!\cdot n_\ell}\tr\left(\rs(X^{(\ell)}) \right).$$
Since $V$ is a $\C$-algebra generated by $\cR$, routine estimates show that $\phi$ is well-defined. It is also clear that $\phi$ has the desired properties.
\end{proof}

For the rest of the paper fix a functional $\phi$ as in Lemma \ref{l:fun}. Then
\begin{equation}\label{e:inner}
(\rs_1,\rs_2)=\phi(\rs_2^*\rs_1)
\end{equation}
is an inner product on $V$. With respect to this inner product we can inductively build an ordered orthogonal basis $\cB$ of $V$ with the property that $\cB\cap V_\ell$ is a basis of $V_\ell$ for every $\ell\in\N$.

\begin{lem}\label{l:fX}
With respect to the inner product \eqref{e:inner} and the ordered basis $\cB$ as above, operators $\fX_1,\dots,\fX_d$ are represented by hermitian matrices in $\cfm$, and $\fX\in\dom_\infty r$.
\end{lem}

\begin{proof}
Since
$$\big(\fX_j\rs_1,\rs_2\big)
= \phi\big(\rs_2^*x_j\rs_1\big)
= \big( \rs_1,\fX_j\rs_2\big)$$
for all $\rs_1,\rs_2\in V$ and $\fX_j(V_\ell)\subseteq V_{\ell+1}$ for all $\ell\in\N$, it follows that the matrix representation of $\fX_j$ with respect to $\cB$ is hermitian and has only finitely many nonzero entries in each column and row.
The rest follows inductively on the construction of $r$ since 
$\fX_j$ are the left multiplication operators on $V$.
\end{proof}

Next we define a complexity-measuring function $\tau:\rex\to\N\cup\{0\}$ as in \cite[Section 4]{KPV}:
\begin{enumerate}[(i)]
	\item $\tau(\alpha)=0$ for $\alpha\in\C$;
	\item $\tau(x_j)=1$ for $1\le j\le d$;
	\item $\tau(s_1+s_2)=\max\{\tau(s_1),\tau(s_2)\}$ for $s_1,s_2\in\rex$;
	\item $\tau(s_1s_2)=\tau(s_1)+\tau(s_2)$ for $s_1,s_2\in\rex$;
	\item $\tau(s^{-1})=2\tau(s)$ for $s,s^{-1}\in\rex$.
\end{enumerate}
Note that $\tau(s^*)=\tau(s)$ for all $s\in\rex$.

\begin{prop}\label{p:gns}
Let the notation be as above, and let $U$ be a finite-dimensional Hilbert space containing $V_{\ell+1}$. If $X$ is a $d$-tuple of hermitian operators on $U$ such that $X\in\hdom r$ and
$$X_j|_{V_\ell}=\fX_j|_{V_\ell}$$
for $j=1,\dots,d$, then $X\in\hdom q$ and
\begin{equation}\label{e:gns}
q(X)\rs=\qq\rs
\end{equation}
for every $q\in R$ and $s\in \overbrace{R\cdots R}^\ell$ satisfying $2\tau(q)+\tau(s)\le \ell+2$.
\end{prop}

\begin{proof}
First note that for every $s\in R\cdots R$,
\begin{equation}\label{e:tau}
\tau(s)\le k\quad \Rightarrow \quad s\in \overbrace{R\cdots R}^k
\end{equation}
since $\tau^{-1}(0)=\C$ and $R\cap\C=\{1\}$. We prove \eqref{e:gns} by induction on the construction of $q$.
If $q=1$ then \eqref{e:gns} trivially holds, and if $q=x_j$ then $\tau(s)\le \ell$ so \eqref{e:gns} holds by \eqref{e:tau}. 
Next, if \eqref{e:gns} holds for $q_1,q_2\in R$ such that $q_1+q_2\in R$ or $q_1q_2\in R$, then it also holds for the latter by the definition of $\tau$ and \eqref{e:tau}. 
Finally, suppose that \eqref{e:gns} holds for $q\in R\setminus\{1\}$ and assume $q^{-1}\in R$. If $2\tau(q^{-1})+\tau(s)\le \ell+2$, then $2\tau(q)+(\tau(q^{-1})+\tau(s))\le \ell+2$.
In particular, $\tau(q^{-1}s)\le \ell$ and so $q^{-1}s\in \overbrace{R\cdots R}^\ell$ by \eqref{e:tau}. Therefore
$$q(X)\qq^{-1}\rs = \qq\qq^{-1}\rs=\rs$$
by the induction hypothesis, and hence $q^{-1}(X)\rs =\qq^{-1}\rs$ since $X\in\hdom q^{-1}$. Thus \eqref{e:gns} holds for $q^{-1}$. 
\end{proof}

\section{Positive noncommutative rational functions}\label{s5}

In this section we prove various positivity statements for noncommutative rational functions. Let $L$ be a hermitian monic pencil of size $e$; that is, $L=I+H_1x_1+\dots+H_dx_d$ with $H_j\in\her{e}$. Then
$$\cD(L) = \bigcup_{n\in\N}\cD_n(L),
\qquad \text{where}\quad\cD_n(L)=\{X\in\her{n}^d\colon L(X)\succeq0 \},$$
is a {\bf free spectrahedron}. The main result of the paper is Theorem \ref{t:main}, which describes noncommutative rational functions that are positive semidefinite or undefined at each tuple in a given free spectrahedron $\cD(L)$. In particular, Theorem \ref{t:main} generalizes \cite[Theorem 3.1]{Pas} to noncommutative rational functions with singularities in $\cD(L)$.

\subsection{Rational convex Positivstellensatz}

Let $L$ be a hermitian monic pencil of size $e$. To $r\in\rex$ we assign the finite set $R$, vector spaces $V_\ell$ and operators $\fX_j$ as in Section \ref{s4}. For $\ell\in\N$ we also define
$$\hV_\ell = \{\rs\in V_\ell\colon \rs=\rs^* \}$$
$$Q_\ell = \left\{
\sum_i \rs_i^*\rs_i+\sum_j \rv_j^* L\rv_j\colon \rs_i \in V_\ell, \rv_j\in V_\ell^e
\right\}\subset \hV_{2\ell+1}.$$
Then $\hV_\ell$ is a real vector space and $Q_\ell$ is a convex cone. The proof of the following proposition is a rational modification of a common argument in free real algebraic geometry (cf. \cite[Proposition 3.1]{HKM} and \cite[Proposition 4.1]{KPV}). A convex cone is salient if it does not contain a line.

\begin{prop}\label{p:closed}
The cone $Q_\ell$ is salient and closed in $\hV_{2\ell+1}$ with the Euclidean topology.
\end{prop}

\begin{proof}
As in the proof of Lemma \ref{l:fun} there exists $X\in\hdom r$ such that
$$\rs(X)\neq0 \qquad \text{for all}\quad \rs\in V_{2\ell+1}\setminus\{0\}.$$
Furthermore, we can choose $X$ close enough to 0 so that $L(X)\succeq \frac12 I$. Then clearly $\rs(X)\succeq0$ for every $\rs\in Q_\ell$, so $Q_\ell\cap-Q_\ell=\{0\}$ and thus $Q_\ell$ is salient. Note that $\|\rs\|_{\bullet}= \|\rs(X)\|$ is a norm on $V_{2\ell+1}$. Also, finite-dimensionality of $\hV_{2\ell+1}$ implies that every element of $Q_\ell$ can be written as a sum of $N=1+\dim\hV_{2\ell+1}$ elements of the form
$$\rs^*\rs\quad \text{and}\quad \rv^* L\rv 
\qquad \text{for}\ \rs\in V_\ell,\ \rv\in V_\ell^e$$
by Carath\'eodory's theorem \cite[Theorem I.2.3]{Bar}. Assume that a sequence $\{\rr_n\}_n\subset Q_\ell$ converges to $\rs\in \hV_{2\ell+1}$. After restricting to a subsequence we can assume that there is $0\le M\le N$ such that
$$\rr_n=\sum_{i=1}^M \rs_{n,i}^*\rs_{n,i}+\sum_{j=M+1}^N \rv_{n,j}^* L\rv_{n,j}$$
for all $n\in\N$. The definition of the norm $\|\cdot\|_\bullet$ implies
$$\|\rs_{n_i}\|_\bullet^2\le \|\rr_n\|_\bullet
\quad \text{and} \quad
\max_{1\le i\le e}\|(\rv_{n_j})_i\|_\bullet^2\le 2\|\rr_n\|_\bullet.$$
In particular, the sequences $\{\rs_{n,i}\}_n\subset V_\ell$ for $1\le i\le M$ and $\{\rv_{n,j}\}_n\subset V_\ell^e$ for $1\le j\le N$ are bounded. Hence, after restricting to subsequences, we may assume that they are convergent: $\rs_i=\lim_n\rs_{n,i}$ for $1\le i\le M$ and $\rv_j=\lim_n\rv_{n,j}$ for $1\le j\le N$. Consequently we have
\[\rs=\lim_n\rr_n=
\sum_{i=1}^M \rs_i^*\rs_i+\sum_{j=M+1}^N \rv_j^* L\rv_j
\in Q_\ell. \qedhere\]
\end{proof}

We are now ready to prove the main result of this paper by combining 
a truncated GNS construction with 
extending matrix tuples into the domain of a rational expression 
as in Proposition \ref{p:main}.

\begin{thm}[Rational convex Positivstellensatz]\label{t:main}
Let $L$ be a hermitian monic pencil and $r\in\rex$. If $Q_{2\tau(r)+1}$ is as above, then $r(X)\succeq0$ for every $X\in\hdom r\cap\cD(L)$ if and only if $\rr\in Q_{2\tau(r)+1}$.
\end{thm}

\begin{proof}
Only the forward implication is nontrivial. Let $\ell=2\tau(r)-2$. If $\rr\neq\rr^*$, then there exists $X\in\hdom r$ such that $\rr(X)\neq\rr(X)^*$. Thus we assume $\rr=\rr^*$. Suppose that $\rr\notin Q_{\ell+3}$. Since $Q_{\ell+3}$ is a salient closed convex cone in $\hV_{2\ell+7}$ by Proposition \ref{p:closed}, there exists a linear functional $\lambda_0:\hV_{2\ell+7}\to\R$ such that $\lambda_0(Q_{\ell+3}\setminus\{0\})=\R_{>0}$ and $\lambda_0(\rr)<0$ by the Hahn-Banach separation theorem \cite[Theorem III.1.3]{Bar}. 
We extend $\lambda_0$ to $\lambda:V_{2\ell+7}\to\C$ as $\lambda(\rs)=\frac12\lambda_0(\rs+\rs^*)+\frac{i}{2}\lambda_0(i(\rs^*-\rs))$.
Then $\langle \rs_1,\rs_2\rangle=\lambda(\rs_2^*\rs_1)$ defines a scalar product on $V_{\ell+3}$. Recall that $\fX_j(V_{\ell+2})\subseteq V_{\ell+3}$. Then for $\rs_1\in V_{\ell+1}$ and $\rs_2\in V_{\ell+2}$,
\begin{equation}\label{e:sym}
\langle \fX_j\rs_1,\rs_2\rangle
 = \lambda\big(\rs_2^*x_j\rs_1\big)
 = \langle \rs_1,\fX_j\rs_2\rangle.
\end{equation}
Furthermore,
\begin{equation}\label{e:L}
\langle L(\fX)\rv,\rv\rangle = \lambda(\rv^*L\rv)>0
\end{equation}
for all $\rv\in V_{\ell+1}^e$, where the canonical extension of $\langle \cdot,\cdot\rangle$ to a scalar product on $\C^e\otimes V_{\ell+1}$ is considered.

Let $\cB$ be an ordered orthogonal basis of $V$ with respect to the inner product $(\cdot,\cdot)$ as in Section \ref{s4}; recall that such a basis has the property that $\cB\cap V_k$ is a basis for $V_k$ for all $k\in\N$. 
Let $\cB_0$ be an ordered orthogonal basis of $V_{\ell+2}$ with respect to $\langle\cdot,\cdot\rangle$ that contains a basis for $V_{\ell+1}$, and let $\cB_1=\cB\setminus V_{\ell+2}$. 
If we identify operators $\fX_j$ with their matrix representations relative to the ordered basis $(\cB_0,\cB_1)$ of $V$, then $\fX_j\in\cfm$ are hermitian matrices by Lemma \ref{l:fX} and \eqref{e:sym}.

Let $U_0$ be the orthogonal complement of $V_{\ell+1}$ in $V_{\ell+2}$ relative to $\langle\cdot,\cdot\rangle$. Since $\fX_j(V_{\ell+1})\subseteq V_{\ell+2}$, we can consider the restriction $\fX_j|_{V_{\ell+1}}$
in a block form
$$\begin{pmatrix}X_j \\Y_j\end{pmatrix}$$
with respect to the decomposition $V_{\ell+2}=V_{\ell+1}\oplus U_0$.
Since $\fX\in\dom_\infty r$,
by Proposition \ref{p:main} there exist a finite-dimensional vector space $U_1$, a scalar product on $V_{\ell+1}\oplus U_0\oplus U_1$ extending $\langle\cdot,\cdot\rangle$, an operator $E$ on $U_0\oplus U_1$, and a $d$-tuple $Z$ of hermitian operators on $U_0\oplus U_1$ 
such that
\begin{equation}\label{e:c1}
\widetilde{X}:=\begin{pmatrix}
X & \begin{pmatrix} Y^* & 0\end{pmatrix}E^* \\
E\begin{pmatrix} Y \\ 0\end{pmatrix} & \mathlarger{Z}
\end{pmatrix}\in\hdom r.
\end{equation}
Since $\fX_j(V_\ell)\subseteq V_{\ell+1}$, we conclude that
\begin{equation}\label{e:c2}
\widetilde{X}_j|_{V_\ell}=\fX_j|_{V_\ell}.
\end{equation}
Observe that for all but finitely many $\ve_1,\ve_2>0$ we can replace $Z,E$ with $\ve_1 Z,\ve_2 E$ and \eqref{e:c1} still holds. By \eqref{e:L} we can thus assume that $Z$ and $E$ are close enough to $0$ so that $L(\widetilde{X})\succeq0$.
Finally, since \eqref{e:c2} holds and $2\tau(r)+\tau(1)=\ell+2$, Proposition \ref{p:gns} implies
$$\langle r(\widetilde{X})1,1\rangle = \langle \rr,1\rangle=\lambda(\rr)<0.$$
Therefore $\widetilde{X}\in\hdom r\cap\cD(L)$ and $r(\widetilde{X})$ is not positive semidefinite.
\end{proof}

Given a unital $*$-algebra $\cA$ and $A=A^*\in\opm_\ell(\cA)$, the {\bf quadratic module} in $\cA$ generated by $A$ is
$$\QM_{\cA}(A) = \left\{
\sum_j v_j^* (1\oplus A)v_j\colon v_j\in \cA^{\ell+1}
\right\}.$$
Theorem \ref{t:main} then in particular states that noncommutative rational functions positive semidefinite on a free spectrahedron $\cD(L)$ belong to $\QM_{\rx}(L)$.

\begin{rem}\label{r:bound}
Let $r\in\rex$ and
$$n=2 \left(e^3 m^2+e m (2 e \ell-1)+\ell (e \ell-2)\right)$$
where $\ell=\dim V_{2\tau(r)-1}$, $m=\dim V_{2\tau(r)}-\dim V_{2\tau(r)-1}$ and $e$ is the size of a linear representation of $r$.
If $r\not\succeq0$ on $\hdom r\cap \cD(L)$, then by Remark \ref{r:size} and the proofs of Theorem \ref{t:main} and Proposition \ref{p:main} there exists $X\in\hdom_n r\cap \cD_n(L)$ such that $r(X)\not\succeq0$.
\end{rem}

The solution of Hilbert's 17th problem for a free skew field is now as follows.

\begin{cor}\label{c:H17}
Let $\rr\in\rx$. Then $\rr\succeq0$ on $\hdom\rr$ if and only if
$$\rr=\rr_1\rr_1^*+\cdots+\rr_m\rr_m^*$$
for some $\rr_i\in\rx$ with $\hdom\rr_i\supseteq \hdom\rr$.
\end{cor}

\begin{proof}
By Proposition \ref{p:rep} there exists $r\in\rr$ such that $\hdom \rr=\hdom r$. The corollary then follows directly from Theorem \ref{t:main} applied to $L=1$ since the hermitian domain of an element in $V_{2\tau(r)}$ contains $\hdom \rr$.
\end{proof}

\begin{rem}
Corollary \ref{c:H17} also indicates a subtle distinction between solutions of Hilbert's 17th problem in the classical commutative context and in the free context. While every (commutative) positive rational function $\rho$ is a sum of squares of rational functions, in general one cannot choose summands that are defined on the whole real domain of the original function $\rho$. On the other hand, a positive noncommutative rational function always admits a sum-of-squares representation with terms defined on its hermitian domain.
\end{rem}

For a possible future use we describe noncommutative rational functions whose invertible evaluations have nonconstant signature; polynomials of this type were of interest in \cite[Section 3.3]{HKV}.

\begin{cor}
Let $\rr=\rr^*\in \rx$. The following are equivalent:
\begin{enumerate}[(i)]
\item there are $n\in\N$ and $X,Y\in\hdom_n \rr$ such that $\rr(X),\rr(Y)$ are invertible and have distinct signatures;
\item neither $\rr$ or $-\rr$ equals $\sum_i\rr_i\rr_i^*$ for some $\rr_i\in\rx$.
\end{enumerate}
\end{cor}

\begin{proof}
(i)$\Rightarrow$(ii) If $\pm\rr=\sum_i\rr_i\rr_i^*$, then $\pm\rr(X)\succeq0$ for all $X\in\hdom\rr$.
\\
(ii)$\Rightarrow$(i) Let $\cO_n=\hdom_n\rr\cap\hdom_n\rr^{-1}$. By Remark \ref{r:dom} there is $n_0\in\N$ such that $\cO_n\neq\emptyset$ for all $n\ge n_0$. Suppose that $\rr$ has constant signature on $\cO_n$ for each $n\ge n_0$, i.e.,  $\rr(X)$ has $\pi_n$ positive eigenvalues for every $X\in\cO_n$. 
Since $\cO_k\oplus\cO_\ell\subset \cO_{k+\ell}$ for all $k,\ell\in\N$, we have
\begin{equation}\label{e:dir}
n\pi_m=\pi_{mn}=m\pi_n
\end{equation}
for all $m,n\ge n_0$. 
If $\pi_{n'}=n'$ for some $n'\ge n_0$, then $\pi_n=n$ for all $n\ge n_0$ by \eqref{e:dir}, so $\rr\succeq 0$ on $\cO_n$ for every $n$. Thus $\rr=\sum_i\rr_i\rr_i^*$ by Theorem \ref{t:main}. Analogous conclusion holds if $\pi_{n'}=0$ for some $n'\ge n_0$. However, \eqref{e:dir} excludes any alternative: if $n_0\le m< n$ and $n$ is a prime number, then $0<\pi_n<n$ contradicts \eqref{e:dir}.
\end{proof}

\subsection{Positivity and invariants}

Let $G$ be a subgroup of the unitary group $\un$. The action of $G$ on $\C^d$ induces a linear action of $G$ on $\rx$. If $G$ is finite and solvable, then the subfield of $G$-invariants $\rx^G$ is finitely generated \cite[Theorem 1.1]{KPPV} and in many cases again a free skew field \cite[Theorem 1.3]{KPPV}. Furthermore, we can now extend \cite[Corollary  6.6]{KPPV} to invariant noncommutative rational functions with singularities.

\begin{cor}\label{c:inv}
Let $G\subset \un$ be a finite solvable group. Then there exists $R_G \in \GL_{|G|}(\rx)$ with the following property. If $\rr\in\rx^G$ and $L$ is a hermitian monic pencil of size $e$, then 
$\rr\succeq0$ on $\hdom \rr\cap \cD(L)$ if and only if $\rr\in\QM_{\rx^G}(L_G)$, where
$$L_G = 
R_G^*R_G\oplus(R_G\otimes I)^*\left(\bigoplus_{g\in G} L^g \right)(R_G\otimes I)
\in \opm_{|G|(e+1)}(\rx^G).$$
\end{cor}

\begin{proof}
Combine \cite[Corollary 6.4]{KPPV} and Theorem \ref{t:main}.
\end{proof}
	
\subsection{Real free skew field and other variations}\label{ss:real}
\def\rerx{\R\plangle x \prangle}
\def\real{{\rm re}}
\def\imag{{\rm im}}

In this subsection we explain how the preceding results apply to real free skew fields and their symmetric evaluations, and to another natural involution on a free skew field.

\begin{cor}[Real version of Theorem \ref{t:main}]\label{c:real}
Let $L$ be a symmetric monic pencil of size $e$ and $\rr\in\rerx$. Then $\rr(X)\succeq0$ for every $X\in\hdom \rr\cap\cD(L)$ if and only if $\rr\in\QM_{\rerx}(L)$.
\end{cor}

\begin{proof}
If $\rr\in\rerx$ and $\rr\succeq0$ on $\hdom\rr\cap\cD(L)$, 
then $\rr\in\QM_{\C\otimes\rerx}(L)$ by Theorem \ref{t:main} because the complex vector spaces $V_\ell$ are spanned with functions given by subexpressions of some $r\in\rr$, and we can choose $r$ in which only real scalars appear. For $\rs\in \C\otimes \rerx$ we define 
$\real (\rs) = \frac12(\rs+\overline{\rs})$ and 
$\imag (\rs) = \frac{i}{2}(\overline{\rs}-\rs)$ in $\rerx$.
If
$$\rr=\sum_j \rs_j^*\rs_j+\sum_k \rv_k^* L\rv_k$$
for $\rs_j\in(\C\otimes \rerx)$ and $\rv_k\in(\C\otimes \rerx)^e$, then
$$
\rr=\real(\rr)=\sum_j \left(\real(\rs_j)^*\real(\rs_j)+\imag(\rs_j)^*\imag(\rs_j)\right)
+\sum_k \left(\real(\rv_k)^*L\,\real(\rv_k)+\imag(\rv_k)^*L\,\imag(\rv_k)\right)
$$
and so $\rr\in\QM_{\rerx}(L)$.
\end{proof}

Given $\rr\in\rerx$ one might prefer to consider only the tuples of real symmetric matrices in the domain of $\rr$, and not the whole $\hdom \rr$. Since there exist $*$-embeddings $\opm_n(\C)\hookrightarrow\opm_{2n}(\R)$, evaluations on tuples of real symmetric $2n\times 2n$ matrices carry at least as much information as evaluations on tuples of hermitian $n\times n$ matrices. Consequently, all dimension-independent statements in this paper also hold if only symmetric tuples are considered. However, it is worth mentioning that for $\rr\in\rerx$, it can happen that $\dom_n \rr$ contains no tuples of symmetric matrices for all odd $n$, e.g. if $\rr=(x_1x_2-x_2x_1)^{-1}$.

\def\rxx{\C\plangle x,x^* \prangle}
\def\rerxx{\R\plangle x,x^* \prangle}

Another commonly considered free skew field with involution is $\rxx$, generated with $2d$ variables $x_1,\dots,x_d,x_1^*,\dots,x_d^*$, which is endowed with the involution $*$ that swaps $x_j$ and $x_j^*$. Elements of $\rxx$ can be evaluated on $d$-tuples of complex matrices. The results of this paper also directly apply to $\rxx$ and such evaluations 
because $\rxx$ is freely generated by elements $\frac12(x_j+x_j^*),\frac{i}{2}(x_j^*-x_j)$ which are fixed by $*$. Finally, as in Corollary \ref{c:real}
we see that a suitable analog of Theorem \ref{t:main} also holds
for $\rerxx$ and evaluations on $d$-tuples of real matrices.

\subsection{Examples of non-convex Positivstellens\"atze}\label{ss:noncvx}

Given $\mm=\mm^*\in\opm_{\ell}(\rx)$ let
$$\cD(\mm) = \bigcup_{n\in\N}\cD_n(\mm),
\qquad \text{where}\quad\cD_n(\mm)=\{X\in\hdom_n \mm \colon \mm(X)\succeq0 \},$$
be its {\bf positivity domain}. Here, the domain of $\mm$ is the intersection of domains of its entries.

\begin{prop}\label{p:change}
Let $\mm=\mm^*\in \opm_{\ell}(\rx)$ and assume there exist a hermitian monic pencil $L$ of size $e\ge \ell$, a $*$-automorphism $\varphi$ of $\rx$, and $A\in\GL_e(\rx)$ such that
\begin{equation}\label{e:change}
\varphi(\mm)\oplus I = A^*LA.
\end{equation}
If $\rr\in\rx$, then $\rr\succeq0$ on $\hdom\rr\cap\cD(\mm)$ if and only if $\rr\in\QM_{\rx}(\mm)$.
\end{prop}

\begin{proof}
The relation \eqref{e:change}, Remark \ref{r:dom} and the convexity of $\cD_n(L)$ imply that the sets $\cD_n(\varphi(\mm))$ and $\cD_n(\mm)$ have the same closures as their interiors in the Euclidean topology for all but finitely many $n$. Therefore
\begin{align*}
\rr|_{\hdom\rr\cap\cD(\mm)}\succeq0 
&\quad\Leftrightarrow\quad
\varphi(\rr)|_{\hdom\varphi(\rr)\cap\cD(\varphi(\mm))}\succeq0 \\
&\quad\Leftrightarrow\quad
\varphi(\rr)|_{\hdom\varphi(\rr)\cap\cD(L)}\succeq0 \\
&\quad\Leftrightarrow\quad
\varphi(\rr)\in\QM_{\rx}(L) \\
&\quad\Leftrightarrow\quad
\rr\in\QM_{\rx}(\varphi^{-1}(L))=\QM_{\rx}(\mm)
\end{align*}
by Theorem \ref{t:main} and \eqref{e:change}.
\end{proof}

The following example presents a family of quadratic noncommutative polynomials $q=q^*\in\pxx$ that admit a rational Positivstellensatz on their (not necessarily convex) positivity domains $\cD(q)=\{X\colon q(X,X^*)\succeq0 \}$.

\begin{exa}\label{ex:posss}
Given a linearly independent set $\{a_0,\dots,a_n\}\subset\spa_\C\{1,x_1,\dots,x_d\}$ let
$$q =a_0^*a_0-a_1^*a_1-\dots-a_n^*a_n \in \pxx.$$
One might say that $q$ is a hereditary quadratic polynomial of positive signature 1.
Note that $\cD_1(q)$ is not convex if $a_0\notin \C$.
Since $a_0,\dots,a_n$ are linearly independent affine polynomials in $\px$ (and in particular $n\le d$), 
there exists a linear fractional automorphism $\varphi$ on $\rx$ such that 
$\varphi^{-1}(x_j)=a_ja_0^{-1}$ for $1\le j \le n$. We extend $\varphi$ uniquely to a $*$-automorphism on $\rxx$.
Then
$$\varphi(a_0)^{-*}\varphi(q)\varphi(a_0)^{-1}=1-x_1^*x_1-\cdots-x_n^*x_n$$
and thus $\varphi(q)\oplus I_n=A^*LA$ where
$$
L=\begin{pmatrix}
1 & x_1^* & \cdots & x_n^* \\
x_1 & \ddots & & \\
\vdots & & \ddots & \\
x_n & & & 1 
\end{pmatrix},
\qquad
A=\begin{pmatrix}
\varphi(a_0) &  & &  \\
-x_1 & 1 & & \\
\vdots & & \ddots & \\
-x_n & & & 1 
\end{pmatrix}
$$
Therefore
$\rr\succeq0$ on $\hdom\rr\cap\cD(q)$ if and only if $\rr\in\QM_{\rxx}(q)$ for every $\rr\in\rxx$ by Proposition \ref{p:change}.

For example, the polynomial $x_1^*x_1-1$ if of the type discussed above, and thus admits a rational Positivstellensatz. In particular,
$$x_1x_1^*-1 = (x_1-x_1^{-*})(x_1^*-x_1^{-1})+x_1^{-*}(x_1^*x_1-1)x_1^{-1} \in \QM_{\rxx}(x_1^*x_1-1).$$
On the other hand, we claim that $x_1x_1^*-1\notin \QM_{\pxx}(x_1^*x_1-1)$ (cf. \cite[Example 4]{HM}). If $x_1x_1^*-1$ were an element of $\QM_{\pxx}(x_1^*x_1-1)$, then the implication
$$S^*S-I\succeq0 \quad\Rightarrow\quad SS^*-I\succeq0$$
would be valid for every operator $S$ on an infinite-dimensional Hilbert space; however it fails if $S$ is the forward shift operator on $\ell^2(\N)$. A different Positivstellensatz (polynomial, but with a slack variable) for hereditary quadratic polynomials is given in \cite[Corollary 4.6]{HKV}.
\end{exa}

\subsection{Eigenvalue optimization}\label{ss:opt}

Theorem \ref{t:main} is also essential for optimization of noncommutative rational functions. Namely, it implies that finding the eigenvalue supremum or infimum of a noncommutative rational function on a free spectrahedron is equivalent to solving a semidefinite program \cite{BPT}. This equivalence was stated in \cite[Section 5.2.1]{KPV} for regular noncommutative rational functions; the novelty is that Theorem \ref{t:main} now confirms its validity for noncommutative rational functions with singularities. 

Let $L$ be a hermitian monic pencil of size $e$, and let $\rr=\rr^*\in\rx$. Suppose we are interested in
$$\mu_* = \sup_{X\in\hdom \rr\cap \cD(L)}\big[\text{maximal eigenvalue of }\rr(X)\big].$$
Choose some $r\in\rr$ (the simpler representative the better) and let $\ell=2\tau(r)+1$. Theorem \ref{t:main} then implies that
\begin{equation}\label{e:sup}
\mu_* = \inf\left\{
\mu\in\R\colon \mu-\rr= 
\sum_{i=1}^M \rs_i^*\rs_i+\sum_{j=1}^N \rv_j^* L\rv_j\colon \rs_i \in V_\ell, \rv_j\in V_\ell^e
\right\}
\end{equation}
where we can take $M=\dim \hV_{2\ell}+1$ and $N=\dim \hV_{2\ell+1}+1$ by Carath\'eodory's theorem \cite[Theorem I.2.3]{Bar}. The right-hand side of \eqref{e:sup} can be stated as a semidefinite program \cite{WSV,BPT}. Concretely, to determine the global (no $L$) eigenvalue supremum of $\rr$, one solves the semidefinite program
\begin{equation}\label{e:sdp}
\begin{split}
\min_H & \quad \mu \\
\text{subject to} & \quad \mu-\rr = \vec{w}^*H\vec{w}, \\
& \quad H\succeq 0
\end{split}
\end{equation}
where $H$ is a $(\dim V_\ell)\times (\dim V_\ell)$ hermitian matrix and $\vec{w}$ is a vectorized basis of $V_\ell$. For constrained eigenvalue optimization ($L$ is present), one can set up a similar semidefinite program using localizing matrices \cite[Definition 1.41]{BKP}.
	
\section{More on domains}\label{s6}

In this section we prove two new results on (hermitian) domains.
One of them is the aforementioned Proposition \ref{p:rep}, which states that every noncommutative rational function admits a representative with the largest hermitian domain. The other one is Proposition \ref{p:canc} on cancellation of singularities of noncommutative rational functions.

\subsection{Representatives with the largest hermitian domain}\label{ss:rep}

We will require a technical lemma about matrices over formal rational expressions and their hermitian domains.
A representative of a matrix $\mm$ over $\rx$ is a matrix over $\rex$ of representatives of $\mm_{ij}$, and the domain of a matrix over $\rex$ is the intersection of domains of its entries.

\begin{lem}\label{l:phd}
Let $m$ be an $e\times e$ matrix over $\rex$ such that $\mm\in\GL_e(\rx)$. Then there exists $s\in\mm^{-1}$ such that $\hdom m\cap\hdom \mm^{-1}=\hdom s$.
\end{lem}

\begin{proof}
Throughout the proof we reserve italic letters ($m,c$, etc.) for matrices over $\rex$ and bold letters ($\mm,\rc$, etc.) for the corresponding matrices over $\rx$.
We prove the statement by induction on $e$. If $e=1$, then $m^{-1}$ is the desired expression. Assume the statement holds for matrices of size $e-1$, and let $c$ be the first column of $m$. 
Then $\hdom c\supseteq\hdom m$ and $c(X)$ is of full rank for every $X\in\hdom m\cap\hdom\mm^{-1}$.
Hence
\begin{equation}\label{e:dom1}
\hdom (c^*c)^{-1} \supseteq\hdom m\cap\hdom\mm^{-1}.
\end{equation}
Let $\widehat{\mm}$ be the Schur complement of $\rc^*\rc$ in $\mm^*\mm$. Note that the entries of $\widehat{\mm}$ are polynomials in entries of $\mm^*\mm$ and $(\rc^*\rc)^{-1}$; lifting these polynomials to formal expressions in $\rex$ we obtain a representative $\widehat{m}\in\widehat{\mm}$ such that
\begin{equation}\label{e:dom2}
\hdom\widehat{m} =\hdom m\cap \hdom (c^*c)^{-1}.
\end{equation}
If $X\in\hdom m$, then $m(X)$ is invertible if and only if $(c^*c)(X)$ and $\widehat{m}(X)$ are invertible.
Thus by \eqref{e:dom1} and \eqref{e:dom2} we have
\begin{equation}\label{e:dom3}
\hdom m \cap \hdom \mm^{-1} = \hdom m \cap (\hdom (c^*c)^{-1}\cap\hdom \widehat{\mm}^{-1}).
\end{equation}
Since $\hat{m}$ is an $(e-1)\times (e-1)$ matrix, by the induction hypothesis there exists $\widehat{s}\in\widehat{\mm}^{-1}$ 
such that $\hdom \widehat{m}\cap\hdom \widehat{\mm}^{-1}=\hdom \widehat{s}$.
By \eqref{e:dom3} we have
\begin{equation}\label{e:dom4}
\hdom m \cap \hdom \mm^{-1} = \hdom m \cap (\hdom (c^*c)^{-1}\cap\hdom \widehat{s}).
\end{equation}
The entries of $(\mm^*\mm)^{-1}$ can be represented by expressions $s'_{ij}$ which are sums and products of expressions $m_{ij},m_{ij}^*,(c^*c)^{-1},\widehat{s}_{ij}$. Thus $s'\in (\mm^*\mm)^{-1}$ satisfies
$$\hdom m \cap \hdom \mm^{-1}=\hdom s'$$
by \eqref{e:dom4}. 
Finally, $s=s' m^*$ is the desired expression because $\mm^{-1}=(\mm^*\mm)^{-1}\mm^*$.
\end{proof}

\begin{proof}[Proof of Proposition \ref{p:rep}]
Let $\rr\in\rx$. Let $e\in\N$, an affine matrix pencil $M$ of size $e$ and $u,v\in\C^e$ be such that $\rr=u^* M^{-1}v$ in $\rx$, and $e$ is minimal. Recall that  $\dom\rr=\bigcup_{r\in\rr}\dom r$. By comparing $(u,M,v)$ with linear representations of representatives of $\rr$ as in \cite[Theorem 1.4]{CR} it follows that
\begin{equation}\label{e:pv}
\dom\rr\subseteq\bigcup_{n\in\N}\left\{X\in\mat{n}^d: \det M(X)\neq0 \right\}.
\end{equation}
Since $M$ contains no inverses, it is defined at every matrix tuple; thus by Lemma \ref{l:phd} there is a representative of $M^{-1}$ whose hermitian domain equals $\{X=X^*\colon \det M(X)\neq0 \}$. Since $\rr$ is a linear combination of the entries in $M^{-1}$, there exists $r\in\rr$ such that $\hdom r=\hdom \rr$ by \eqref{e:pv}.
\end{proof}

\begin{exa}\label{ex:hdom}
The domain of $\rr\in\rx$ given by the expression $(x_4-x_3x_1^{-1}x_2)^{-1}$ equals
$$\dom\rr = \bigcup_{n\in\N} \left\{X\in\mat{n}^4 \colon
\det \begin{pmatrix}
X_ 1 & X_2 \\ X_3 & X_4
\end{pmatrix}\neq0
 \right\}$$
and $\dom r\subsetneq \dom \rr$ for every $r\in\rr$ by \cite[Example 3.13]{Vol}.

Following the proof of Proposition \ref{p:rep} and Lemma \ref{l:phd} let $\mm=\left(\begin{smallmatrix}x_1 & x_2 \\ x_3 & x_4\end{smallmatrix}\right)$.
Then
$$\mm^*\mm=\begin{pmatrix}x_1^2+x_3^2 & x_1x_2+x_3x_4 \\ x_2x_1+x_4x_3 & x_2^2+x_4^2\end{pmatrix}$$
and the Schur complement of $\mm^*\mm$ with respect to the $(1,1)$-entry equals
$$\widehat{\mm}=x_2^2+x_4^2-(x_2x_1+x_4x_3)(x_1^2+x_3^2)^{-1}(x_1x_2+x_3x_4).$$
Since
$$\mm^{-1}=(\mm^*\mm)^{-1}\mm^*=
\begin{pmatrix}
\star & \star \\
-\widehat{\mm}^{-1}(x_2x_1+x_4x_3)(x_1^2+x_3^2)^{-1} & \widehat{\mm}^{-1}
\end{pmatrix}
\begin{pmatrix}
\star & x_3 \\
\star & x_4
\end{pmatrix}$$
and
$$\rr = \begin{pmatrix}0 & 1 \end{pmatrix}\mm^{-1}\begin{pmatrix} 0\\ 1\end{pmatrix}$$
we conclude that the formal rational expression
$$\left(x_2^2+x_4^2-(x_2x_1+x_4x_3)(x_1^2+x_3^2)^{-1}(x_1x_2+x_3x_4)\right)^{-1}
\left(x_4-(x_2x_1+x_4x_3)(x_1^2+x_3^2)^{-1}x_3\right)$$
represents $\rr$, and its hermitian domain coincides with $\hdom \rr$. Of course, the expression $(x_4-x_3x_1^{-1}x_2)^{-1}$ is a much simpler representative of $\rr$.
\end{exa}

\subsection{Cancellation of singularities}

In the absence of left ideals in skew fields, the following proposition serves as a rational analog of Bergman's Nullstellensatz for noncommutative polynomials \cite[Theorem 6.3]{HM}. The proof below omits some of the details since it is a derivate of the proof of Theorem \ref{t:main}.

\begin{prop}\label{p:canc}
The following are equivalent for $\rr,\rs\in\rx$.
\begin{enumerate}[(i)]
\item $\ker\rr(X)\subseteq\ker\rs(X)$ for all $X\in\dom\rr\cap\dom\rs$;
\item $\dom (\rs\rr^{-1})\supseteq \dom\rr\cap\dom\rs$.
\end{enumerate}
\end{prop}

\begin{proof}
(ii)$\Rightarrow$(i) If (ii) holds, then $\rs(X)=(\rs(X)\rr(X)^{-1})\rr(X)$ for every $X\in\dom\rr\cap\dom\rs$, and so $\ker\rr(X)\subseteq\ker\rs(X)$.

(i)$\Rightarrow$(ii) Suppose (ii) does not hold; thus there are $r\in\rr$, $s\in\rs$ and $Y\in\dom r\cap\dom s$ such that $\det r(Y)=0$. Similarly as in Section \ref{s4} denote
$$R=\{1\}\cup \{q\in\rex\setminus\C\colon q \text{ is a subexpression of }r \text{ or }s  \}$$
and let $\cR$ be its image in $\rx$. We also define finite-dimensional vector spaces $V_\ell$ and the finitely generated algebra $V$ as before. The left ideal $V\rr$ in $V$ is proper: if $\qq\rr=1$ for $q\in V$, then $q(Y)r(Y)=I$ since $Y\in\dom q$, which contradicts $\det r(Y)=0$. Furthermore, $\rs\notin V\rr$ since (ii) does not hold. Let $K=V/V\rr$, and let $K_\ell$ be the image of $V_\ell$ for every $\ell\in\N$. Let $\fX_j: K\to K$ be the operator given by the left multiplication with $x_j$; note that $\fX_j(K_\ell)\subseteq K_{\ell+1}$ for all $\ell$. By induction on the construction of $q\in R$ it is straightforward to see that $q(\fX)$ is well-defined for every $q\in R$. Let $\ell=2\max\{\tau(r),\tau(s)\}-2$. By Proposition \ref{p:main2} there exist a finite-dimensional vector space $U$ and a $d$-tuple of operators $X$ on $K_{\ell+1}\oplus U$ such that
$X\in\dom r\cap\dom s$ and
$$X_j|_{K_\ell}=\fX_j|_{K_\ell}$$
for $j=1,\dots,d$. A slight modification of Proposition \ref{p:gns} implies that
$$r(X)[1]=[\rr]=0,\qquad s(X)[1]=[\rs]\neq0$$
where $[\qq]\in K$ denotes the image of $\qq\in V$.
\end{proof}

The implication (i)$\Rightarrow$(ii) in Proposition \ref{p:canc} fails if only hermitian domains are considered (e.g. take $\rr=x_1^2$ and $\rs=x_1$). It is also worth mentioning that while Proposition \ref{p:canc} might look rather straightforward at first glance, there is a certain subtlety to it. Namely, the equivalence in Proposition \ref{p:canc} fails if only matrix tuples of a fixed size are considered. For example, let $\rr=x_1$ and $\rs=x_1x_2$;
then $\dom_1\rr\cap\dom_1\rs=\C^2$ and $\ker\rr(X)\subseteq\ker\rs(X)$ for all $X\in \C^2$, but
$\dom_1 (\rs\rr^{-1})=\C\setminus\{0\}\times\C$ (cf. \cite[Example  2.1 and Theorem 3.10]{Vol}).


\end{document}